\documentclass{article}

\usepackage{amsmath}
\usepackage{amsthm, amssymb}
\usepackage{url}
\newtheorem{thm}{Theorem}
\newtheorem{cor}{Corollary}
\newtheorem*{ident*}{Identity}

\begin{document}

\title{Proofs of Power Sum and Binomial Coefficient Congruences Via Pascal's Identity}
\author{Kieren MacMillan and Jonathan Sondow}
\date{}
\maketitle

\begin{abstract}
A well known and frequently cited congruence for power sums is
\begin{equation*}
	1^n + 2^n + \dotsb + p^n \equiv
	\begin{cases}
		\, -1 \pmod{p} \, \text{ if }\, p-1 \mid n, \\
		\, \hspace{0.75em} 0 \pmod{p} \, \text{ if } \, p-1 \nmid n,
	\end{cases}
\end{equation*}
where $n\ge1$ and $p$ is prime. We survey the main ingredients in several known proofs. Then we give an elementary proof, using an identity for power sums proven by Pascal in $1654$. An application is a simple proof of a congruence, due to Hermite and Bachmann, for certain sums of binomial coefficients.
\end{abstract}

\noindent In the literature on \emph{power sums}
\begin{equation}
	S_n(a) :=  \sum_{j=1}^a j^n = 1^n + 2^n + \dotsb + a^n \qquad (n \ge 0,\, a\ge1), \label{EQ: def sigma}
\end{equation}
the following congruence is well known and is frequently cited.

\begin{thm}  \label{THM: HW119}
Let $p$ be a prime. For $n\ge1$, we have
\begin{align*}
	S_n(p) \equiv
	\begin{cases}
		\, -1 \pmod{p} \, \textit{ if }\, p-1 \mid n, \\
		\, \hspace{0.75em} 0 \pmod{p} \, \textit{ if } \, p-1 \nmid n.
	\end{cases}
\end{align*}
\end{thm}

For example, it is used to prove theorems on Bernoulli numbers (that of von Staudt-Clausen in \cite{carlitz}, \cite[Theorem~118]{hw}, of Carlitz-von Staudt in \cite[Theorem~4]{carlitz}, \cite{tophat}, and of Almkvist-Meurman in \cite[Theorem~9.5.29]{cohen}) and to study the Erd\H{o}s-Moser Diophantine equation \mbox{$S_n(m-1)=m^n$} (in \cite{mpim}, \cite{tophat}, \cite{Moser}) as well as other exponential Diophantine equations and Stirling numbers of the second kind (in~\cite{kellner}).

A component of most proofs of Theorem~\ref{THM: HW119} is Fermat's little theorem \cite[Theorem~71]{hw}, \cite[p.~36]{rose}, which says that if $p$ is prime, then
\begin{equation}
     p \nmid a \; \Longrightarrow \; a^{p-1} \equiv 1\pmod{p}. \label{flt}
\end{equation}

To prove the nontrivial case $p-1 \nmid n$, another component is needed. The usual proof \cite[Theorem~119]{hw}, \cite[Lemma~1]{mpim} relies on the theory of primitive roots. (It gives an integer $g$ such that \mbox{$g^n\not\equiv 1 \pmod{p}.$} Then $p\nmid g$, implying $S_n(p) \equiv g^n S_n(p) \pmod{p}$, and we infer that $p\mid S_n(p).$) Another proof \cite[Lemma~2]{tophat}, due to Zagier, invokes Lagrange's theorem (see \cite[p.~39]{rose}) on roots of polynomials over $\mathbb{Z}/p\mathbb{Z}$. (Using it, Zagier deduces the existence of an integer~$g$ with $g^n\not\equiv 1 \pmod{p}.$) Still a third proof~\cite{carlitz} employs Bernoulli numbers and finite differences.

In this note, we give a very elementary proof of both cases of Theorem~\ref{THM: HW119}, using a recurrence for the sequence of power sums $S_0(a),S_1(a), \dots$ proven by Pascal~\cite{Pascal} in $1654$ (see \cite[p.~82]{edwards}).\\

\noindent {\bf  Pascal's Identity.}
{\it If $n\ge0$ and $a\ge1$, then}
\begin{align}
	\sum_{k=0}^n \binom{n+1}{k} S_k(a)
		= (a+1)^{n+1} - 1. \label{EQ: P identity}
\end{align}

\begin{proof}
For $j>0$, the binomial theorem gives
\begin{align*}
	(j+1)^{n+1} - j^{n+1} = \sum_{k=0}^n \binom{n+1}{k} j^k.
\end{align*}
Summing from $j=1$ to $a$, the left-hand side telescopes to $(a+1)^{n+1} - 1$, and by \eqref{EQ: def sigma} we get the desired identity.
\end{proof}

\begin{proof}[Proof of Theorem~\ref{THM: HW119}]
The case $p-1 \mid n$ follows easily from \eqref{flt}.

To prove the second case, suppose on the contrary that $p-1 \nmid n$ (so $p>2$) but $p \nmid S_n(p)$. Let $n$ be the smallest such number and write $n=d(p-1)+r$, where $d\ge0$ and $0 < r < p-1$. Now \eqref{flt} yields $S_n(p) \equiv S_r(p)\pmod{p}$, and the minimality of $n$ implies first that $n = r$ and then that $p \mid S_k(p)$ for $k  < n$ (note that $S_0(p)=p$). Hence \eqref{EQ: P identity} with  $a=p$ implies $p\mid (n+1) S_n(p)$. But then prime $p>n+1$ forces $p \mid S_n(p)$, a~contradiction. This completes the proof.
\end{proof}

As a bonus, Pascal's identity allows a simple proof of a congruence for certain sums of binomial coefficients $\binom{m}{k}$ (generalizing the easily-established facts that $\sum_{k=1}^{m-1} \binom{m}{k}$ is even for $m>0$, and that if prime $p\le m\le 2(p-1)$, then $p$~divides $\binom{m}{p-1}$). The case $m$ odd is due to Hermite \cite{hermite} in 1876, and the general case to Bachmann \cite[p.~46]{bachmann} in 1910; for these and related results, see \cite[pp.~270--275]{dickson}.

\begin{cor}[\textbf{Hermite and Bachmann}] \label{COR: P mod p}
If $m>0$ and $p$ is prime, then
\begin{align} \label{EQ: hb}
	\sum_{\substack{0 \, <\, k\, < \, m,\\ p-1\, \mid\, k}} \hspace{-.4em}\binom{m}{k}
		\equiv 0 \pmod{p},
\end{align}
where the sum is over all $k \equiv 0 \pmod{p-1}$ with $1\le k \le m-1$.
\end{cor}

\begin{proof}
Set $n=m-1$ and $a=p$ in \eqref{EQ: P identity}, then reduce modulo $p$. Using Theorem~\ref{THM: HW119}, the result follows.
\end{proof}

For example,
\begin{align*}
	\binom{14}{4} + \binom{14}{8} + \binom{14}{12} = 1001+ 3003+ 91 \equiv 0 \pmod{5}.
\end{align*}
 
For \eqref{EQ: hb} and generalizations due to Glaisher and Carlitz, see \cite[p.~70, Lemma 9.5.28; p.~133, Exercise~62; and p.~327, Proposition~11.4.11]{cohen}. Recently, Dilcher~\cite{dilcher} discovered an analog of \eqref{EQ: hb} for \emph{alternating} sums.

\paragraph{Acknowledgments.} We thank the three referees for several suggestions and references, and Pieter Moree for sending us a preprint of \cite{tophat}.

\bigskip

\noindent\textit{55 Lessard Avenue, Toronto, Ontario, Canada M6S 1X6\\
kieren@alumni.rice.edu}

\bigskip

\noindent\textit{209 West 97th Street, New York, NY 10025\\
jsondow@alumni.princeton.edu}


\begin{thebibliography}{99}
\bibitem{bachmann} P.\ Bachmann, \textit{Niedere Zahlentheorie}, Part 2, Teubner, Leipzig, 1910; Parts 1 and 2 reprinted in one volume, Chelsea, New York, 1968.
\bibitem{carlitz} L.\ Carlitz, The Staudt-Clausen Theorem, \textit{Math. Mag.} \textbf{34} (1961) 131--146.
\bibitem{cohen} H.\ Cohen, \textit{Number Theory, Volume II: Analytic and Modern Tools}, Graduate Texts in Mathematics, vol. 240, Springer-Verlag, New York, 2007.
\bibitem{dickson} L.\ E.\ Dickson, {\it History of the Theory of Numbers}, vol. 1, Carnegie Institution of Washington, Washington, D. C., 1919; reprinted by Dover, Mineola, NY, 2005.
\bibitem{dilcher} K.\ Dilcher, Congruences for a class of alternating lacunary sums of binomial coefficients, \textit{J. Integer Seq.} \textbf{10} (2007) Article 07.10.1.
\bibitem{edwards} A.\ W.\ F.\ Edwards, \textit{Pascal's Arithmetical Triangle}, Charles Griffin, London, 1987.
\bibitem{hw} G.~H.~Hardy and E.~M.~Wright, \emph{An Introduction to the Theory of Numbers}, D.~R.~Heath-Brown and J.~H.~Silverman, eds., 6th ed., Oxford University Press, Oxford, 2008.
\bibitem{hermite} Ch.\ Hermite, Extrait d'une lettre $\rm\grave{a}$ M. Borchardt, \textit{J. Reine Angew. Math.} \textbf{81} (1876) 93--95.
\bibitem{kellner} B.\ C.\ Kellner, The equivalence of Giuga's and Agoh's conjectures (2004), available at \url{http://arxiv.org/abs/math/0409259}.
\bibitem{mpim} P.\ Moree, Moser's mathemagical work on the equation \mbox{$1^k + 2^k +\cdots+\ (m-1)^k = m^k$} (2009), available at \url{http://www.mpim-bonn.mpg.de/preprints/send?bid=4096}.
\bibitem{tophat} ---------, A top hat for Moser's four mathemagical rabbits, \emph{Amer. Math. Monthly} (to appear).
\bibitem{Moser} L.\ Moser, On the Diophantine equation $1^n + 2^n + \dotsb + (m - 1)^n = m^n$, \emph{Scripta Math.} \textbf{19} (1953) 84--88.
\bibitem{Pascal} B.\ Pascal, Sommation des puissances num\'{e}riques, in \emph{Oeuvres compl$\grave{e}$tes}, vol.\ III, Jean Mesnard, ed., Descl\'{e}e-Brouwer, Paris, 1964, 341--367; English translation by A.\ Knoebel, R.\ Laubenbacher, J.\ Lodder, and D.~\ Pengelley, Sums of numerical powers, in \emph{Mathematical Masterpieces: Further Chronicles by the Explorers}, Springer Verlag, New York, 2007, 32--37.
\bibitem{rose} H.\ E.\ Rose, \textit{A Course in Number Theory}, Clarendon Press, Oxford, 1988.
\end{thebibliography}
\end{document}